\documentclass[12pt]{article}

\usepackage{enumerate}
\usepackage{amsmath}
\usepackage{amssymb,latexsym}
\usepackage{amsthm}
\usepackage{color}

\usepackage{graphicx}

\usepackage{hyperref}

\usepackage{amscd}
\usepackage[normalem]{ulem}

\DeclareMathOperator{\N}{N}

\newtheorem{theorem}{Theorem}[section]

\newtheorem{corollary}[theorem]{Corollary}
\newtheorem{definition}[theorem]{Definition}
\newtheorem{proposition}[theorem]{Proposition}

\newtheorem{remark}[theorem]{Remark}

\newcommand{\Fq}{\mathbb F_q}
\newcommand{\Fqn}{\mathbb F_{q^n}}

\newcommand{\F}{{\mathbb F}}

\newcommand{\Tr}{\hbox{{\rm Tr}}}

\newcommand{\la}{\langle}

\newcommand{\AG}{\mathrm{AG}}

\newcommand{\PG}{\mathrm{PG}}

\newcommand{\ran}{\rangle_{\Fqn}}
\newcommand{\raq}{\rangle_{\Fq}}

\title{A condition for scattered linearized polynomials involving Dickson matrices}
\author{Corrado Zanella}
\date{}

\begin{document}
\maketitle

\begin{abstract}
A linearized polynomial over $\F_{q^n}$ is called scattered when
for any $t,x\in\F_{q^n}$, the condition $xf(t)-tf(x)=0$ holds if and only if
$x$ and $t$ are $\Fq$-linearly dependent.
General conditions for  linearized polynomials over $\F_{q^n}$ to be scattered
can be deduced from the recent results in \cite{Cs2018,CsMPZ2019,MGS2019,PZ2019}.
Some of them are based on the Dickson matrix associated with a linearized polynomial.
Here a new condition involving Dickson matrices is stated.
This condition is then applied to the Lunardon-Polverino binomial 
$x^{q^s}+\delta x^{q^{n-s}}$, allowing to prove that for any $n$ and $s$,
if $\N_{q^n/q}(\delta)=1$, then
the binomial is not scattered.
Also, a necessary and sufficient condition for $x^{q^s}+bx^{q^{2s}}$ to be scattered is
shown which is stated in terms of a special plane algebraic curve.
\end{abstract}



\section{Introduction}
A point $P$ of the projective space $\PG(d-1,q^n)$ is a one-dimensional subspace of the
vector space $\Fqn^d$; that is,
$P=\langle v\ran=\{cv\colon c\in\Fqn\}$ for some nonzero $v\in\Fqn^d$.

Let $U$ be an $r$-dimensional $\Fq$-subspace of $\Fqn^d$.
Then
\[
L_U:=\{\langle v\ran\colon v\in U,v\neq0\} \]
is an \emph{$\Fq$-linear set} (or just \emph{linear set}) of \emph{rank} $r$ in $\PG(d-1,q^n)$.
Let $u,v\in U$. If $u=cv$, $c\in\Fq$, then clearly $\langle u\ran=\langle v\ran$.
If  this is the only case in which two vectors of $U$ determine the same point of $\PG(d-1,q)$,
that is,  $\langle v\ran=\langle u\ran$ if and only if
$\langle v\raq=\langle u\raq$, then $L_U$ is called a \emph{scattered} linear set.
Equivalently, $L_U$ is scattered if and only if it has maximum size $(q^r-1)/(q-1)$ 
with respect to $r$.
The linear sets are related to combinatorial objects, such as
blocking sets, two-intersection sets, finite semifields, rank-distance codes, and many
others.
The interested reader is referred to the survey by O. Polverino \cite{Po2010} and to
\cite{Sh}, where J. Sheekey builds a bridge with the rank-distance codes. 

Assume that in particular $U$ is an $\Fq$-subspace of $\Fqn^2$, $\dim_{\Fq}U=n$.
In this case 
$L_U:=\{\langle v\ran\colon v\in U,v\neq0\}\subseteq\PG(1,q^n)$,
is called a \emph{maximum} linear set of $\PG(1,q^n)$, since by the dimension formula
any linear set of rank greater than $n$ equals $\PG(1,q^n)$.
Up to projectivities of $\PG(1,q^n)$ it may be assumed that $\la(0,1)\ran\not\in L_U$.
Hence 
\[ L_U={L_f}=\{\langle(x,f(x))\ran\colon x\in\Fqn^*\} \]
where $f(x)$ is a suitable $\Fq$-linear map, that is a linearized polynomial: 
\begin{equation}\label{linpol}
f(x)=\sum_{i=0}^{n-1}a_ix^{q^i},\quad a_i\in\Fqn,\quad i=0,1,\ldots,n-1.
\end{equation}
If $L_f$ is scattered, then $f(x)$ is called a
\emph{scattered} linearized polynomial, or \emph{scattered $q$-polynomial} with respect to $n$. 
A property characterizing the scattered $q$-polynomials is that
for any $x,y\in\Fqn^*$, ${f(x)}/x={f(y)}/y$ if and only if
$\langle x\raq=\langle y\raq$.

A first example of scattered $q$-polynomial is $f(x)=x^q$ \cite{BL2000}, with respect to any $n$. 
Indeed, for any   $x,y\in\Fqn^*$, ${f(x)}/x={f(y)}/y$, is equivalent to
$x^{q-1}=y^{q-1}$, hence to $x/y\in\Fq^*$.
A derived example is $f(x)=x^{q^s}$, $\gcd(n,s)=1$. Indeed $(x/y)^{q^s-1}=1$ implies
$x/y\in\mathbb F_{q^s}\cap\Fqn^*=\Fq^*$.
In both cases above, $L_f=\{\langle(x,f(x))\ran\colon x\in\Fqn^*\}=
\{\langle(1,z)\ran\colon z\in\Fqn,\,\N_{q^n/q}(z)=1\}$,
where $\N_{q^n/q}(z)=z^{(q^n-1)/(q-1)}$ denotes the norm over $\Fq$ of $z\in\Fqn$.
The related linear set is called a
\emph{linear set of pseudoregulus type}.

The next example has been given by G. Lunardon and O. Polverino \cite{LP2001} and generalized
in \cite{LMPT,Sh}:
\[f(x)=x^{q^s}+\delta x^{q^{n-s}},\quad n\ge4,\quad \gcd(n,s)=1,\quad \N_{q^n/q}(\delta)\neq1.\]
In particular cases, the condition $\N_{q^n/q}(\delta)\neq1$ has been proved to be necessary for $f(x)$ to be
scattered  \cite{BartoliZhou,CsZ2018,LMPT,ZZ}.
In section \ref{twobin} it will proved that actually it is necessary for any $n$ and $s$.
Further examples of scattered $q$-polynomials are given in \cite{CsMZ2018,CMPZ,MMZ,ZZ}.
All of them are with respect to $n=6$ or $n=8$.
D. Bartoli, M. Giulietti, G. Marino, and O. Polverino \cite{BGMP2018} proved that if
$\hat f(x)$ is the adjoint of $f(x)$ with respect to
the bilinear form $\langle x,y\rangle=\Tr_{q^n/q}(xy)$ in $\Fqn^2$,
where $\Tr_{q^n/q}(z)=\sum_{i=0}^{n-1}z^{q^i}$ denotes the trace over $\Fq$ of $z\in\Fqn$,
 then $L_f=L_{\hat f}$.
This implies that if the polynomial $f(x)$ in (\ref{linpol}) is scattered, then also
$\hat f(x)=\sum_{i=0}^{n-1}a_i^{q^{n-i}}x^{q^{n-i}}$ is.
Up to the knowledge of the author of this paper, no more examples of scattered $q$-polynomials
are known.
So, it would seem that scattered $q$-polynomials are rare.
D. Bartoli and Y. Zhou \cite{BartoliZhou} formalized such an idea of scarcity by proving
that the pseudoregulus and Lunardon-Polverino polynomials are, roughly speaking,
the only $q$-polynomials of a certain type which are scattered for infinitely many $n$. 

Recently, a great deal of effort has been put in finding conditions for
$q$-polynomials to be scattered \cite{Cs2018,CsMPZ2019,MGS2019,PZ2019}.
Some of them are based on the \emph{Dickson matrix} associated
with the $q$-polynomial in (\ref{linpol}), that is, the $n\times n$ matrix
\[
M_{q,f}=
\begin{pmatrix}
a_0&a_1&a_2&\cdots&a_{n-1}\\
a_{n-1}^q&a_0^q&a_1^q&\cdots&a_{n-2}^q\\
a_{n-2}^{q^2}&a_{n-1}^{q^2}&a_0^{q^2}&\cdots&a_{n-3}^{q^2}\\
\vdots&&&&\vdots\\
a_{1}^{q^{n-1}}&a_{2}^{q^{n-1}}&a_3^{q^{n-1}}&\cdots&a_{0}
^{q^{n-1}}\end{pmatrix}.
\]
It is well-known that the rank of $M_{q,f}$ equals the rank of $f(x)$, see for example
\cite[Proposition 4.4]{WL2003}.
This rank can be computed by applying the following result by B.\ Csajb\'ok:
\begin{theorem}[{\cite[Theorem 3.4]{Cs2018}}]\label{Th:Sottobicchiere}
Let $M_{q,f}$ be the Dickson matrix associated with the $q$-polynomial in (\ref{linpol}).
Denote by $M_{q,f}^{(r)}$ the $r\times r$ submatrix of $M_{q,f}$ obtained by considering the last 
$r$ columns and the first $r$ rows of $M_{q,f}$.
Then the rank of $f(x)$ is $t$ if and only if
$\det (M_{q,f}^{(n)}) =\det (M_{q,f}^{(n-1)}) =\cdots = \det (M_{q,f}^{(t+1)}) =0$, 
and $\det (M_{q,f}^{(t)}) \neq 0$.
\end{theorem}
A $q$-polynomial $f(x)\in\Fqn[x]$ is scattered if and only if for any $m\in\Fqn$ the dimension of 
the kernel  of $f_m(x)=mx+f(x)$ is at most one.
So, by Theorem \ref{Th:Sottobicchiere} a necessary and sufficient condition for $f(x)$ 
 to be scattered is that the system of two equations
\[ \det (M_{q,f_m}^{(n)})=\det (M_{q,f_m}^{(n-1)})=0 \]
has no solution in the variable $m\in\Fqn$.

In this paper a condition consisting of one equation (Proposition \ref{zero})
is proved, and applied to two binomials.
It would seem that one equation is better than two in order to prove that a given $q$-polynomial 
$f (x)$ is not scattered, while two equations will usually be 
more helpful in the proof that $f (x)$ is.
As a matter of fact, here the condition $\N_{q^n/q}(\delta)\neq1$ is proved to be necessary for the 
Lunardon-Polverino binomial to be scattered 
(cf.\ Theorem \ref{cnecess}).
Furthermore,
two necessary and sufficient conditions for $x^{q^s}+bx^{q^{2s}}$ (where $\gcd(s,n)=1$) 
to be scattered are 
stated in Propositions \ref{pqq} and \ref{geomalg}.
This leads to the fact that the polynomial $x^q+bx^{q^2}$, $b\neq0$, is never scattered
if $n\ge5$ (cf.\ Proposition \ref{223} and Remark \ref{maria}).

\section{A condition for scattered linearized polynomials}

In this paper $s$, $n$, $q$ and $\sigma$ will always denote natural numbers
such that $n\ge3$, $\gcd(s,n)=1$, $q$ is the power of a prime and $\sigma=q^s$.
Any $\Fq$-linear endomorphism of $\Fqn$ can be represented in the form
\begin{equation}\label{sigmapol}
f(x)=a_0x+a_1x^\sigma+a_2x^{\sigma^2}+\cdots+a_{n-1}x^{\sigma^{n-1}}\in\Fqn[x].
\end{equation}
As a matter of fact, if $\tau$ is the permutation $i\mapsto is$ of $\mathbb Z/(n)$,
then $f(x)$ is the same function of $\tilde f(x)=\sum_{i=0}^{n-1}a_{\tau^{-1}(i)}x^{q^i}$.
Generalizing the notion  of Dickson matrix given in the previous section,
the \emph{$\sigma$-matrix of Dickson} associated with the linearized polynomial
$g(t)=\sum_{i=0}^{n-1}a_it^{\sigma^i}$
is
\[
M_{\sigma,g}=\begin{pmatrix}
a_0&a_1&a_2&\cdots&a_{n-1}\\
a_{n-1}^\sigma&a_0^\sigma&a_1^\sigma&\cdots&a_{n-2}^\sigma\\
a_{n-2}^{\sigma^2}&a_{n-1}^{\sigma^2}&a_0^{\sigma^2}&\cdots&a_{n-3}^{\sigma^2}\\
\vdots&&&&\vdots\\
a_{1}^{\sigma^{n-1}}&a_{2}^{\sigma^{n-1}}&a_3^{\sigma^{n-1}}&\cdots&a_{0}
^{\sigma^{n-1}}\end{pmatrix}.
\]
This is just the Dickson matrix $M_{q,\tilde g}$ associated with $\tilde g(t)$
after a permutation of the row and columns.
Indeed, the element in row $r$ and column $c$ of $M_{q,\tilde g}$,
$r,c\in\{0,1,\ldots,n-1\}$, is $m_{rc}=a_{\tau^{-1}(c-r)}^{q^r}=a_{\tau^{-1}(c)-\tau^{-1}(r)}^{q^r}$.
By applying $\tau$ to both the row and column index,
$m_{\tau(r)\tau(c)}=a_{c-r}^{\sigma^r}$ follows.
Therefore, the rank of $M_{\sigma,g}$ equals the rank of $g(t)$.
\begin{remark}\label{rem-dickson}
Each row of an $n$-order  $\sigma$-matrix of Dickson is obtained from the previous one
(cyclically)
by the map \[\phi:(X_0,X_1,\ldots,X_{n-1})\mapsto(X_{n-1},X_0,\ldots,X_{n-2})^\sigma\]
which is an invertible semilinear map of $\F_{q^n}^n$ into itself.
\end{remark}

The polynomial (\ref{sigmapol}) is scattered if and only if $f_1(x)=\sum_{i=1}^{n-1}a_ix^{\sigma^i}$
is.
Hence in the following $a_0$ will always be zero.
\begin{proposition}\label{zero}
Let $f(x)=\sum_{i=1}^{n-1}a_ix^{\sigma^i}$ be a linearized polynomial over $\F_{q^n}$,
and \[g(t)=g_x(t)=-f(x)t+\sum_{i=1}^{n-1}a_ix^{\sigma^i}t^{\sigma^i}=-f(x)t+f(xt).\]
Then the following conditions are equivalent:
\begin{enumerate}[(i)]
\item the polynomial $f(x)$ is scattered;
\item for any $x\in\F_{q^n}^*$, a nonsingular $(n-1)$-order minor of $M_{\sigma,g}$ exists;
\item for any $x\in\F_{q^n}^*$, all $(n-1)$-order minors of $M_{\sigma,g}$ are nonsingular.
\end{enumerate}
\end{proposition}
\begin{proof}
The polynomial $f(x)$ is scattered if and only if for any $x\in\F_{q^n}^*$
the rank of $h(t)=xf(t)-tf(x)$ is $n-1$, that is, the rank of
\[
M_{\sigma,h}=\begin{pmatrix}-f(x)&a_1x&a_2x&\ldots&a_{n-1}x\\
a_{n-1}^\sigma x^\sigma&-f(x)^\sigma&a_1^\sigma x^\sigma&\ldots&a_{n-2}^\sigma x^\sigma\\
a_{n-2}^{\sigma^2} x^{\sigma^2}&a_{n-1}^{\sigma^2}x^{\sigma^2}&-f(x)^{\sigma^2}&
\ldots&a_{n-3}^{\sigma^2} x^{\sigma^2}\\ \vdots&&&&\vdots\\
a_1^{\sigma^{n-1}}x^{\sigma^{n-1}}&a_2^{\sigma^{n-1}}x^{\sigma^{n-1}}&a_3^{\sigma^{n-1}}
x^{\sigma^{n-1}}&\ldots&-f(x)^{\sigma^{n-1}}\end{pmatrix}
\]
is always $n-1$.
By dividing the rows of $M_{\sigma,h}$ by $x$, $x^\sigma$, $x^{\sigma^2}$, $\ldots$,
$x^{\sigma^{n-1}}$, respectively, and then multiplying the columns for that same elements, 
one obtains
\[
\begin{pmatrix}
-f(x)&a_1x^\sigma&a_2x^{\sigma^{2}}&\ldots&a_{n-1}x^{\sigma^{n-1}}\\
a_{n-1}^\sigma x&-f(x)^\sigma&a_1^{\sigma} x^{\sigma^2}&\ldots&a_{n-2}^\sigma x^{\sigma^{n-1}}\\
a_{n-2}^{\sigma^2} x&a_{n-1}^{\sigma^2}x^{\sigma}&-f(x)^{\sigma^2}&
\ldots&a_{n-3}^{\sigma^2} x^{\sigma^{n-1}}\\ \vdots&&&&\vdots\\
a_1^{\sigma^{n-1}}x&a_2^{\sigma^{n-1}}x^{\sigma}&a_3^{\sigma^{n-1}}
x^{\sigma^{2}}&\ldots&-f(x)^{\sigma^{n-1}}\end{pmatrix},
\]
that is, the matrix $M_{\sigma,g}$.
By Remark \ref{rem-dickson}, if a 
$\sigma$-matrix of Dickson is singular, then any row is a linear combination
of the remaining ones.
Hence the rank of $M_{\sigma,g}$ equals the rank of any $(n-1)\times n$ matrix obtained
from it by deleting a row.
Furthermore, since the sum of the columns of $M_{\sigma,g}$ is zero, all
$(n-1)$-order minors have the same rank of $M_{\sigma,g}$.
\end{proof}

\section{Two linearized  binomials}\label{twobin}

\begin{definition}
For any $\delta\in\F_{q^n}$,
\[
f_{\sigma,\delta}(x)=x^\sigma+\delta x^{\sigma^{n-1}}
\]
is the \emph{Lunardon-Polverino binomial}.
\end{definition}
If $\N_{q^n/q}(\delta)\neq1$, then $f_{\sigma,\delta}$ is scattered
\cite{LMPT,LP2001,LTZ,Sh}.

\begin{proposition}\label{p-LP}
The polynomial $f_{\sigma,\delta}(x)$ is scattered if  only if
there is no $x\in\F_{q^n}^*$ such that
\begin{equation}\label{e-LP}
\sum_{i=0}^{n-1}z^{(\sigma^i-1)/(\sigma-1)}=0,
\end{equation}
where
$z=\delta x^{\sigma^{n-1}-\sigma}$.
\end{proposition}
\begin{proof}
The $(n-1)$-th order
North-West principal minor  of the $\sigma$-matrix of Dickson associated with
the polynomial 
\[g(t)=-f_{\sigma,\delta}(x)t+\sum_{i=1}^{n-1}a_ix^{\sigma^i}t^{\sigma^i}=
-f_{\sigma,\delta}(x)t+x^\sigma t^\sigma+\delta
x^{\sigma^{n-1}}t^{\sigma^{n-1}},\]
further normalized row by row, is
\begin{equation}\label{e-Mz}
B(z)=\begin{pmatrix}
-(1+z)&1&0&0&\cdots&0&0\\
z^\sigma&-(1+z)^\sigma&1&0&\cdots&0&0\\
0&z^{\sigma^2}&-(1+z)^{\sigma^2}&1&\cdots&0&0\\
\vdots&&&&&&\vdots\\
0&0&0&0&\cdots&-(1+z)^{\sigma^{n-3}}&1\\
0&0&0&0&\cdots&z^{\sigma^{n-2}}&-(1+z)^{\sigma^{n-2}}
\end{pmatrix}.
\end{equation}
By Laplace expansion along the last column and induction on $n$, the determinant of $B(z)$ 
can be computed as $(-1)^{n+1}\sum_{i=0}^{n-1}z^{(\sigma^i-1)/(\sigma-1)}$.
\end{proof}
The following can be useful in understanding the role of $\delta$:
\begin{proposition}\label{rephr}
Let $z\in\Fqn$. Then (\ref{e-LP}) holds if and only if
there exists a $y\in\Fqn^*$ such that $z=y^{\sigma-1}$ and $\Tr_{q^n/q}(y)=0$.
\end{proposition}
\begin{proof}
Any solution of (\ref{e-LP}) is nonzero.
Raising $\sum_{i=0}^{n-1}z^{(\sigma^i-1)/(\sigma-1)}$ to the $\sigma$, multiplying by $z$ and then 
subtracting to the original equation yields $1-\N_{q^n/q}(z)=0$.
So, $z$ is a solution of (\ref{e-LP}) if and only if
$z=y^{\sigma-1}$ for some $y\in\Fqn^*$, and $\sum_{i=0}^{n-1}y^{\sigma^i-1}=0$.
The latter equation is equivalent to $\Tr_{q^n/q}(y)=0$.
\end{proof}
Propositions \ref{p-LP} and \ref{rephr} together show that, 
if $f_{\sigma,\delta}$ is not scattered, then there is an $x\in\Fqn$ such that 
$\N_{q^n/q}(\delta x^{\sigma^{n-1}-\sigma})=\N_{q^n/q}(\delta)
\N_{q^n/q}(x^{\sigma^{n-1}-\sigma})=1$.
On the other hand, $q-1$ divides $\sigma^{n-1}-\sigma$, hence
 $\N_{q^n/q}( x^{\sigma^{n-1}-\sigma})=1$.
Summarizing, if
$\N_{q^n/q}(\delta)\neq1$, then the Lunardon-Polverino binomial is scattered,
as is known.

\begin{theorem}\label{cnecess}
If $\N_{q^n/q}(\delta)=1$, then
the Lunardon-Polverino binomial $f_{\sigma,\delta}(x)$ is not scattered.
\end{theorem}
\begin{proof}
\underline{Case odd $n$.}
Since \[x^{\sigma^{n-1}-\sigma}=(x^\sigma)^{\sigma^{n-2}-1}\] and 
$\gcd(s(n-2),n)=1$, the expression $x^{\sigma^{n-1}-\sigma}$ takes all
values in $\F_{q^n}$ whose norm over $\Fq$ is equal to one.
This allows the substitution $\delta x^{\sigma^{n-1}-\sigma}=w^{\sigma-1}$
into (\ref{e-LP}).
So, $f_{\sigma,\delta}(x)$ is not scattered if and only if
$\Tr_{q^n/q}(w)=0$ for some nonzero $w$ and this is trivial.

\underline{Case even $n$.}
Since $\gcd(\sigma^{n-2}-1,\sigma^n-1)=\sigma^2-1$,
the set of all powers of elements in $\Fqn$ with exponent $\sigma^{n-2}-1$
coincides with the set of all powers with exponent $\sigma^{2}-1$.
Hence for any $x\in\F_{q^n}$ there exists $u\in\F_{q^n}$ such that
$x^{\sigma^{n-1}-\sigma}=u^{\sigma^2-1}$, and conversely.
This allows the substitution $z=\delta u^{\sigma^2-1}$ in (\ref{e-LP}),
meaning that if there is $u$ such that 
\begin{equation}\label{lpriform}
\sum_{i=0}^{n-1}\left(\delta u^{\sigma^2-1}\right)^{(\sigma^i-1)/(\sigma-1)}=0,
\end{equation}
then $f_{\sigma,\delta}(x)$ is not scattered.
So, taking $\delta=d^{\sigma-1}$,
(\ref{lpriform}) is equivalent to
$\Tr_{q^n/q}(du^{\sigma+1})=0$, $u\neq0$.
This is a quadratic form in $u$ in a vector space over $\Fq$ of dimension
greater than two which has at least one nontrivial zero.
\end{proof}

The theorem above has been proved in the particular cases
$n=4$ in \cite{CsZ2018}, $s=1$ in \cite{BartoliZhou},
both $n$ and $q$ odd in \cite{LMPT}, and odd $n$ in \cite{ZZ}.

\begin{proposition}\label{pqq}
The polynomial $f(x)=x^\sigma+bx^{\sigma^2}$ is scattered if only if
there is no $x\in\F_{q^n}^*$ such that
\begin{equation}\label{e-pqq}
\sum_{i=0}^{n-1}w^{(\sigma^i-1)/(\sigma-1)}=0,\ \mbox{ where }
w=-(1+b^{-1}x^{\sigma-\sigma^2}).
\end{equation}
\end{proposition}
\begin{proof}
The $\sigma$-matrix of Dickson associated with
the polynomial \[g(t)=-f(x)t+x^\sigma t^\sigma+bx^{\sigma^2}t^{\sigma^2},\]
further normalized by dividing the rows by $bx^{\sigma^2}$, $b^{\sigma}x^{\sigma^3}$, $\ldots$ is
\[
A=\begin{pmatrix}
w&-(1+w)&1&\cdots&0&0\\
0&w^\sigma&-(1+w)^\sigma&\cdots&0&0\\
0&0&w^{\sigma^2}&\cdots&0&0\\
\vdots&&&&&\vdots\\
1&0&0&\cdots&w^{\sigma^{n-2}}&-(1+w)^{\sigma^{n-2}}\\
-(1+w)^{\sigma^{n-1}}&1&0&\cdots&0&w^{\sigma^{n-1}}
\end{pmatrix}.
\]

The matrix  obtained by deleting the last row and first 
column is $B(w)$ (cf. (\ref{e-Mz})).
\end{proof}
\begin{corollary}\label{corq}
Assume $b_1,b_2\in\F_{q^n}$ and $\N_{q^n/q}(b_1)=\N_{q^n/q}(b_2)$.
Then the polynomials $f_i(x)=x^\sigma+b_ix^{\sigma^2}$, $i=1,2$, are either both scattered,
or both non-scattered.
\end{corollary}
\begin{proof}
If the norm of $b_1$ is zero then the statement is trivial, so assume that it is not.
Define $w_i(x)=-(1+b_i^{-1}x^{\sigma-\sigma^2})$ for $i=1,2$, and note that
$w_1(x)=w_2(y)$ is equivalent to
$b_1/b_2=((x/y)^\sigma)^{\sigma-1}$, that is, $((x/y)^\sigma)^{\sigma-1}=c^{\sigma-1}$
for some $c\in\Fqn^*$.
This equation can be always solved in both $x$ and $y$, whence $w_1(x)$ and $w_2(y)$ take the same
set of values.
\end{proof}
\begin{remark}
Corollary \ref{corq} allows to look at only $q-1$ linearized polynomials, given $s$, $n$, and $q$.
This makes a computer search easier.
Computations with GAP\footnote{Code: \href{https://pastebin.com/pgTXX76C}{https://pastebin.com/pgTXX76C}.}
 show that there are no scattered linearized polynomials of the form
$l_b(x)=x^q+bx^{q^2}$, $b\neq0$, for any $q<223$ if $n=5$.
In \cite{MZ2019} it is proved that for $n=5$ and $q\ge223$ the linearized polynomial
$l_b(x)$ is not scattered for any $b\neq0$.
The next proposition summarizes this.
\end{remark}
\begin{proposition}\label{223}
If $n=5$ and $b\in\F_{q^5}^*$, then the $q$-polynomial $l_b(x)=x^q+bx^{q^2}\in\F_{q^5}[x]$ is 
non-scattered.
\end{proposition}
\begin{remark}
For $n=4$ there are scattered polynomials of type $l_b(x)$, $b\neq0$. By the results in 
\cite{CsZ20162,CsZ2018}, all the related linear sets are of Lunardon-Polverino type, up to 
collineations.
\end{remark}
\begin{proposition}\label{geomalg}
Let $b\in\Fqn^*$.
The polynomial $x^\sigma+bx^{\sigma^2}\in\Fqn[x]$ is not scattered if and only if
the algebraic curve $b^{-1}X^{q-1}+Y^{\sigma-1}+1=0$ in $\AG(2,q^n)$ has a point
$(x_0,y_0)$ with coordinates in $\Fqn^*$, such that $\Tr_{q^n/q}(y_0)=0$.
\end{proposition}
\begin{proof}
By Proposition \ref{rephr},
the first equation in (\ref{e-pqq}) is equivalent to the existence of $y\in\Fqn^*$
such that $w=y^{\sigma-1}$, $\Tr_{q^n/q}(y)=0$.
The second equation $y^{\sigma-1}+1+b^{-1}x^{q-q^2}=0$ has solutions with
$x\neq0$ if and only if $b^{-1}x^{q-1}+y^{\sigma-1}+1=0$ does.
\end{proof}
\begin{remark}\label{maria}
Very recently, M. Montanucci \cite{pc} proved that if $n>5$,
then for any $q$ the algebraic curve $b^{-1}X^{q-1}+Y^{q-1}+1=0$ has a point with 
the properties above.
Together with Propositions \ref{223} and \ref{geomalg} this implies that for $n\ge5$ 
no $q$-polynomial of type $l_b(x)=x^q+bx^{q^2}$, $b\neq0$, is scattered.
\end{remark}

\bigskip

\noindent
Corrado Zanella\\
Dipartimento di Tecnica e Gestione dei Sistemi Industriali\\
Universit\`a degli Studi di Padova\\
Stradella S. Nicola, 3\\
36100 Vicenza VI\\
Italy\\
\emph{corrado.zanella@unipd.it}

\end{document}